\documentclass{amsart}
\usepackage{amsmath}

\usepackage{amssymb}
\usepackage{amscd}
\usepackage{enumerate}

\renewcommand{\d}{\mathrm{d}}

\newcommand{\Bad}{\mathrm{Bad}}
\newcommand{\Reals}{\mathbb{R}}
\newcommand{\Int}{\mathbb{Z}}
\newcommand{\rd}{\mathbb{R}^d}
\newcommand{\RN}{\mathbb{R}^N}
\newcommand{\RL}{\mathbb{R}^L}
\newcommand{\RH}{\mathbb{R}^H}

\newcommand{\Nat}{\mathbb{N}}
\newcommand{\e}{\mathbf{e}}
\newcommand{\x}{\mathbf{x}}
\newcommand{\X}{\mathbf{X}}
\newcommand{\y}{\mathbf{y}}
\newcommand{\Y}{\mathbf{Y}}
\newcommand{\z}{\mathbf{z}}

\newcommand{\R}{\mathbb{R}}
\newcommand{\Z}{\mathbf{Z}}
\newcommand{\A}{\mathbf{A}}
\newcommand{\B}{\mathbf{B}}
\newcommand{\ubf}{\mathbf{u}}
\newcommand{\p}{\mathbf{p}}
\newcommand{\q}{\mathbf{q}}
\newcommand{\w}{\mathbf{w}}
\newcommand{\vbf}{\mathbf{v}}
\newcommand{\xnm}{\|\x\|}
\newcommand{\ynm}{\|\y\|}
\newcommand{\Anm}{\|{\mathcal{A}}(\X)\|}
\newcommand{\Bnm}{\|{\mathcal{B}}(\Y)\|}

\newcommand{\cl}{\mathcal{L}}
\newcommand{\cm}{\mathcal{M}}
\newcommand{\cy}{{\mathcal{Y}}}
\newcommand{\cz}{{\mathcal{Z}}}
\newcommand\ssm{\setminus}

\newcommand {\ignore}[1]  {}

\newcommand\nz{\setminus \{0\}}
\newcommand{\dist}{\operatorname{dist}}
\newcommand{\eM}{\operatorname{M}}
\newcommand{\0}{\mathbf{0}}
\newcommand{\op}{\mathrm{op}}

\newcommand{\supp}{\operatorname{supp}}
\newcommand{\df}{{\, \stackrel{\mathrm{def}}{=}\, }}

\newtheorem{thm}{Theorem}[section]
\newtheorem{lem}[thm]{Lemma}
\newtheorem{prop}[thm]{Proposition}

\newtheorem{lemma}[thm]{Lemma}
\newtheorem{cor}[thm]{Corollary}
\newtheorem{claim}[thm]{Claim}

\newtheorem{pro}[thm]{Proposition}

\theoremstyle{definition}
\newtheorem {defn}{Definition}
\newtheorem{remark}{Remark}
\newtheorem{exa}[thm]{Example}
\newtheorem{observation}{Observation}

\begin{document}
\title{Badly Approximable Systems of Affine Forms and Incompressibility on Fractals}
\author{Ryan Broderick, Lior Fishman, and David Simmons} 
\address{Department of Mathematics, Northwestern University
2033 Sheridan Road Evanston, IL 60208-2730} 
\email{ryan@math.northwestern.edu}
\address{Department of Mathematics, University of North Texas
1155 Union Circle 311430
Denton, TX 76203-5017}
\email{lfishman@unt.edu, DavidSimmons@my.unt.edu}
\begin{abstract}
We explore and refine techniques for estimating the Hausdorff
dimension of exceptional sets and their diffeomorphic images used in \cite{BFKRW}, 
which were motivated by the work of W. Schmidt \cite{S1}, D. Kleinbock,
E. Lindenstrauss and B. Weiss \cite{KLW} and C. McMullen \cite{Mc}.
Specifically, we use a variant of Schmidt's game to
deduce the strong $C^1$ incompressibility of
the set of badly approximable systems of linear forms as well as of the set of vectors which are badly approximable with
respect to a fixed system of linear forms.
This generalizes results in \cite{BFK}, \cite{ET}, and \cite{BFKRW}.
\end{abstract}
\maketitle

\section{Introduction}

Fix integers $M,N\geq 1$ and let $\eM_{M\times N}$ denote the set of $M\times N$ matrices.
If $A\in\eM_{M\times N}$ and $\x\in \Reals^M$, then the pair $(A,\x)$ defines an affine transformation $\q\mapsto A\q - \x$ from $\R^N$ to $\R^M$.
The components of this affine transformation can be regarded as a system of $M$ affine forms in $N$ variables.
This system is said to be \emph{badly approximable} if
\[
\inf_{\q\in\Int^N\nz}\|\q\|^{N/M}\dist(A\q - \x,\Int^M) > 0\,,
\]
where $\|\cdot\|$ represents the Euclidean norm
and $\dist(\x, \y)$ is the distance in the associated metric.
Note that the above inequality is equivalent to the existence of some
$c > 0$ such that for all $\q \in \Int^N \setminus \{0\}$ and $\p \in \Int^M$,
\[
\|(A\q - \x) - \p\| > c\|\q\|^{-N/M}.
\]
We write $\Bad(M,N)$ to denote the set of all pairs $(A,\x)$ whose corresponding system of affine forms is badly approximable. We also define the slices of $\Bad(M,N)$:
\begin{align*}
\Bad_A(M,N) &\df \big\{ \x \in \Reals^M :(A,\x)\in \Bad(M,N)\big\}\\
\Bad_{\x}(M,N) &\df \big\{ A \in \eM_{M\times N} :(A,\x)\in \Bad(M,N)\big\}.
\end{align*}

In \cite{K2}, D.\ Kleinbock proved that
$\Bad(M,N)$ has full Hausdorff dimension. 
In \cite{BHKV}, Y.\ Bugeaud, S.\ Harrap, S.\ Kristensen, and S.\ Velani
improved this result by showing that for any $A\in\eM_{M\times N}$ the slice $\Bad_A(M,N)$ has dimension $M$.
This latter set was later shown to be winning in the sense of Schmidt's game (see Section \ref{game}),
first in the case $M=N=1$ by J.\ Tseng in \cite{T2} and then in general by N.\ Moshchevitin
in \cite{M}. This winning property implies that $\Bad_A(M,N)$ exhibits rather remarkable behavior
under intersections, namely it is \textsl{incompressible}, a term introduced by S.\ G.\ Dani in
\cite{D}.

\begin{defn}
\label{incompdef}
A set $S \subseteq \rd$ is said to be \textsl{incompressible} if, for each nonempty open $U \subseteq \rd$, each $L \geq 1$, and each sequence $f_i: U\to\rd$ of $L$-bi-Lipschitz maps,
\begin{equation}
\label{incompeq}
\dim\left(\bigcap_{i=1}^{\infty} f_i^{-1}(S)\right) = d.
\end{equation}
\end{defn}

Here, a map $f:U\to\rd$ is called $L$-bi-Lipschitz if for each $\x,\y\in U$,
$$L^{-1}\|\x-\y\| \le \|f(\x)-f(\y)\| \le L \|\x-\y\|.$$
If $K \subseteq \rd$, we will say a set $S$ is \textsl{incompressible on $K$}
if, under the conditions of Definition \ref{incompdef}, the set
\begin{equation}
\label{intersectionwithK}
K \cap \bigcap_{i=1}^{\infty} f_i^{-1}(S)
\end{equation}
has the same Hausdorff dimension as $K \cap U$, whenever $U \cap K \neq \varnothing$.
Incompressibility, as Dani defined it, is then simply the case $K = \rd$.
In \cite{BFK} and \cite{ET} independently, Moshchevitin's result was improved by showing
that $\Bad_A(M,N)$ is winning on certain fractal subsets of $\Reals^M$, i.e. those which support absolutely decaying measures (see \cite{KLW}).

In \cite{BFKRW}, it was shown that for every $M$ the set $\Bad_{\0}(M,1)$ is hyperplane absolute winning (HAW), i.e. it is winning for a certain variant of Schmidt's game (see Section \ref{game}). In particular, HAW sets are winning for Schmidt's game played on hyperplane diffuse sets. Note in particular that the support of an absolutely decaying measure is hyperplane diffuse (Theorem 5.1 in \cite{BFKRW}).

It was also shown in \cite{BFKRW} that the HAW property implies that a set satisfies a variant of incompressibility. Namely, if a set $S$ satisfies (\ref{incompeq}) for every sequence
of $C^1$ maps, then $S$ is said to be \textsl{strongly $C^1$ incompressible}. Here the restriction on the sequence has been both relaxed and tightened, since the bi-Lipschitz constants are no longer required to be uniform, but differentiability is now assumed.
% In that paper however it was
%necessary to add a smoothness condition in order to drop the uniformity for the sets considered, 
%so $S$ was said to be
%\textsl{strongly $C^1$ incompressible} if (\ref{incompeq}) holds
%when $f_i : U \to \rd$ is a sequence of nonsingular $C^1$ 
%diffeomorphisms\footnote{Note that strong $C^1$ incompressibility is implied by strong incompressibility, since any %nonsingular $C^1$ maps $f_i$ will, 
%by the continuity of their first derivatives,
%be bi-Lipschitz on compact subsets of $U$ and $U$ can be written as a union of countably
%many of its compact subsets.}.
If a set $S$ is HAW, then it is strongly $C^1$ incompressible on hyperplane diffuse sets (see \cite{BFKRW}).

Using this terminology, we are able to strengthen the above-mentioned result
from \cite{BFK} and \cite{ET} in the following way.
\begin{thm}
\label{BadAHAW}
For every $M,N\in\Nat$ and $A\in \eM_{M\times N}$, $\Bad_A(M,N)$ is HAW.
\end{thm}
Applying Theorems 2.3, 2.4, 4.6, 4.7, and 5.3 from \cite{BFKRW}, we immediately deduce the following corollary:
\begin{cor}
\label{BadAincompressible}
Fix $M,N\in\Nat$ and $A\in \eM_{M\times N}$, and let $S = \Bad_A(M,N)$.
For any hyperplane diffuse $K \subseteq \Reals^M$, and for any sequence of $C^1$ diffeomorphisms $(f_i)_{i = 1}^\infty$, the set (\ref{intersectionwithK}) is winning on K. In particular (\ref{intersectionwithK}) has positive dimension, and full dimension if $K$ is the support of an Ahlfors regular measure. Thus, $\Bad_A(M,N)$ is strongly $C^1$ incompressible on $K$.
\end{cor}

One may also consider a slice in the other factor, i.e. the set $\Bad_{\x}(M,N)$ defined above. In the special case $\x = \0$, this set is called the set of \textsl{badly approximable systems of linear forms}.
This set was shown to be winning, and hence incompressible, by Schmidt in \cite{S2}.
Using similar techniques, the second-named author later proved in \cite{F2} that it is winning on sets which are the support of absolutely friendly measures (see \cite{KLW}). In the present note, we further adapt Schmidt's original proof idea in order to improve this result. Below we prove the following.

\begin{thm}
\label{maintheorem}
For every $M,N\in\Nat$, $\Bad_{\0}(M,N)$ is HAW.
\end{thm}

\noindent In particular, using the same theorems from \cite{BFKRW} it is easy to deduce a corollary to this theorem similar to Corollary \ref{BadAincompressible}.

By inserting some additional steps into the winning strategy of \cite{F2}, M.\ Einsiedler and J.\ Tseng proved in \cite{ET} that in fact $\Bad_\x(M,N)$ is winning on the supports of absolutely friendly measures for all $\x \in \Reals^M$, not just $\0$. In fact, by combining the argument from \cite{ET} with the proof of Theorem \ref{maintheorem}, one can show that $\Bad_{\x}(M,N)$ is HAW:

\begin{cor}
\label{maintheoremgeneral}
For every $M,N\in\Nat$ and $\x\in\Reals^M$, $\Bad_{\x}(M,N)$ is HAW.
\end{cor}

\noindent We omit the proof for concision.\\

A natural question is whether $\Bad_{\0}(M,N)$ is incompressible on hyperplane diffuse sets. We answer the question in the negative in Example \ref{counterexample}.

The structure of the paper is as follows. In Section \ref{game}, we discuss Schmidt's game and variations thereof. In Section \ref{measures}, we discuss hyperplane diffuse sets and Ahlfors regular measures. We provide a proof of Theorem \ref{BadAHAW} in Section \ref{BadAproof}. Sections \ref{outline}-\ref{sectionhardtheoremproof} are devoted to the proof of Theorem \ref{maintheorem}, starting with an outline of the proof in Section \ref{outline}.\\

{\bf{Acknowledgments}}: The collaboration of this paper began during a workshop held in October 2011 at the University of North Texas Mathematics Department, which was partially sponsored by the RTG in Logic and Dynamics program through the NSF grant DMS-0943870. We would also like to thank Dmitry Kleinbock and Barak Weiss for reading the preliminary version of this paper and making important suggestions.

\ignore{

\noindent For the rest of this paper, 
if $U,V\in\mathbb{R}^{D}$ then $\left|U\right|$ 
is the usual vector length, i.e., $(\Sigma _{i=1}^{D}u_i^{2})^{\frac{1}{2}}$ and
$U\cdot V$ is the standard inner product. \\

}

\section{Schmidt's game and the hyperplane game}
\label{game}
In \cite{S1}, W.\ Schmidt introduced a game referred to thereafter as Schmidt's game,
and used it to define a property of subsets of a complete metric space, the $\alpha$-winning
property, which is stable under countable intersections and often
gives a lower bound on Hausdorff dimension. The game has proven to be a useful
tool for estimating dimension, due to the countable intersection property and its
stability under certain transformations. We will define the game
only on closed subsets of $\rd$ endowed with the Euclidean metric, as we will
only play the game on these spaces. This will allow us to simplify the presentation somewhat.

Let $K$ be any closed subset of $\rd$.
For any $0 < \alpha, \beta < 1$, the $(\alpha,\beta)$-game is played by two players, 
whom we will call Bob and Alice, 
who take turns choosing balls in $\rd$ whose centers lie in $K$, with Bob moving first. The players must play so as to satisfy
\[
B_1 \supseteq A_1 \supseteq B_2 \supseteq \dots,
\]
and
\begin{equation}
\label{radii}
\rho(A_i) = \alpha\rho(B_i)\text{ and }
\rho(B_{i+1}) = \beta\rho(A_i)\text{ for }i\in\Nat,
\end{equation}
where $B_i$ and $A_i$ are Bob's and Alice's $i$th moves, respectively, and where $\rho(B)$ is the radius of $B$.
Since the sets $B_i \cap K$ form a nested
sequence of nonempty, closed subsets of $K$ whose diameters tend to zero, it follows that $\bigcap_i B_i$ contains a single point, which must lie in $K$.
A set $S \subseteq X$ is said to be {\sl $(\alpha,\beta)$-winning on K\/} if Alice has a strategy guaranteeing that 
\begin{equation}
\label{Alicewins}
\bigcap_{i=1}^\infty B_i \subseteq S
\end{equation}
regardless of the way Bob chooses to play. It is said to be {\sl $\alpha$-winning\/} if it is $(\alpha,\beta)$-winning for every $0 < \beta < 1$, and {\sl winning\/} if it is $\alpha$-winning for some $\alpha$.

For each $\alpha > 0$, the class of $\alpha$-winning subsets of a given set $K$ 
is closed under countable intersection (see \cite{S1}). Furthermore, if $K$ is the support of an Ahlfors regular measure, then every winning set has full Hausdorff dimension (see \cite{F}). Taken together, these two
properties make Schmidt's game very useful in providing a lower bound on the dimension
of certain subsets of $\rd$: 
If a set in $\rd$ is naturally written as a countable intersection $S = \bigcap_i S_i$, then
bounding the dimension of each $S_i$ doesn't allow one to say anything about $S$, but
proving that each $S_i$ is $\alpha$-winning immediately implies that $\dim(S) = d$. 

Furthermore, many fractal subsets of $\rd$
prove to be hospitable playgrounds for the game -- namely,
hyperplane diffuse sets (see Section \ref{measures}). When the game is
played on such sets, there is a uniform lower bound
on the dimension of winning sets (see \cite{BFKRW}). Examples of hyperplane diffuse sets
include many well-known fractals, e.g. the middle-thirds Cantor set
and the Sierpinski triangle.

The winning property is thus very useful. It arises naturally
in both dynamics and Diophantine approximation \cite{Dani-rk1, D,
Dani conference, ET, Fae, FPS, KW2, KW3, M, S1, S2, T1, T2}.
%Furthermore, several variants of Schmidt's game have been considered which lead to a stronger winning property.
%Alice is able to direct the play of the game into the target set even with a set of rules
%more favorable to Bob, leading to a stronger winning property.
%See \cite{Mc} for definitions of the strong winning and absolute winning properties.
We will be interested in a variant of Schmidt's game, the hyperplane game, defined in \cite{BFKRW}. Sets which are winning for the hyperplane game are called hyperplane absolute winning (HAW). For technical reasons we will also consider a third variant of Schmidt's game, which we shall call the hyperplane percentage game. As we shall see, in the hyperplane percentage game the rules are more favorable to the Alice, and so any HAW set is automatically winning for the hyperplane percentage game as well. We will show that the converse also holds (Lemma \ref{PHAWequivalence}). Thus, to show that a set is HAW, we will often find it more convenient to describe a strategy for Alice in the hyperplane percentage game.

Fix an integer $0 \le k \le d-1$,
parameters $0 < \beta < 1/3$ and $0 < p < 1$, and a target set $S$.
The $k$-dimensional $(\beta,p)$-game is defined as follows:
Bob begins as usual by choosing a closed ball $B_1 \subseteq \rd$.
Then, for each $i \geq 1$, once $B_i$ is chosen,
Alice chooses a finite sequence of $k$-dimensional affine subspaces $(\cl_{i,j})_{j = 1}^{N_i}$ and a finite sequence of numbers $(\varepsilon_{i,j})_{j = 1}^{N_i}$ satisfying $0 < \varepsilon_{i,j} \leq \beta\rho(B_i)$. Here $N_i$ can be any positive integer that Alice chooses. We will denote by $\cl_{i,j}^{(\varepsilon)}$ the $\varepsilon$-thickening of $\cl_{i,j}$. Bob then must choose a ball $B_{i+1} \subseteq B_i$
with $\rho(B_{i+1}) \geq \beta\rho(B_i)$
such that 
\[
B_{i+1} \cap \cl_{i,j}^{(\varepsilon_{i,j})}= \varnothing\text{ for at least }pN_i\text{ values of }j.
\]
That is, at each stage of the game, Alice chooses any number of neighborhoods of affine subspaces
she wants and Bob must choose his next ball disjoint from at least $p N_i$ of these.
Thus we obtain
as before a nested sequence of closed sets $B_1 \supseteq B_2 \supseteq \dots$
and declare Alice the winner if and only if $S \cap \bigcap_i B_i \neq \varnothing$.
Note that $\bigcap_i B_i$ need not be a single point in this game, since the radii
$\rho(B_i)$ are not forced to $0$.
If Alice has a strategy to win regardless of Bob's play, we say that $S$
is \textsl{$k$-dimensionally $(\beta,p)$-winning}.
If there exist $\beta_0 > 0$ and $0 < p < 1$ such that
$S$ is $k$-dimensionally $(\beta,p)$-winning for each
$0 < \beta < \beta_0$, we say that $S$ is
\textsl{$k$-dimensionally percentage winning}. In the case $k = d-1$,
we will say that $S$ is hyperplane percentage winnning (HPW) and call this the HPW property\footnote{In
fact, this is the only version of the game we will consider; we include the more general definition
to be consistent with the analogous $k$-dimensional absolute winning properties defined
in \cite{BFKRW}.}.

The $k$-dimensional absolute winning property can now be defined easily as a strengthening of
the one above. If, for some $0 \le k \le d-1$ and $0 < \beta < 1/3$, Alice has a strategy to win the above game while always choosing $N_i = 1$ (the value of $p$ does not matter), we say $S$ is {\textsl{$k$-dimensionally $\beta$-absolute winning}}, and if there exists a $\beta_0$ such that $S$ is 
$k$-dimensionally $\beta$-absolute winning for all $0 < \beta < \beta_0$, we
say that $S$ is {\textsl{$k$-dimensionally absolute winning}}. In the case $k = 0$, we simply say that $S$ is \emph{absolute winning}; such sets were considered by C. McMullen \cite{Mc}. In the case
$k = d-1$, we say that $S$ is {\textsl{hyperplane absolute winning}} (HAW). This definition agrees with the one given in \cite{BFKRW}.

Note that for large values of $\beta$, it is possible for Alice to leave Bob with no available moves after finitely
many turns. In \cite{BFKRW}, where the HAW game is defined on arbitrary closed subsets $K$ 
of $\rd$, this situation was resolved by proclaiming Bob the winner. It was noted there however
that if $K$ satisfies a geometric condition called hyperplane diffuseness, 
then for sufficiently small $\beta$ this situation will never arise.\footnote{To see this for $\rd$, in the greater generality of the HPW game, suppose Bob chooses a move at random.
Then the expected number of neighborhoods that he will intersect is less than $C N_i\beta$ for some constant $C$. Thus by Markov's inequality, the probability that he will intersect at least $(1 - p)N_i$ of the neighborhoods is less than $C N_i\beta/[(1 - p)N_i] = C\beta/(1 - p)$. For $\beta$ sufficiently small this is strictly less than one, and so it is possible for Bob to intersect fewer than $(1 - p)N_i$ of the neighborhoods.} Since we will only play on $\rd$, proclaiming Alice the winner instead will not affect the class of sets which are HPW or HAW. In our proofs we will use this modified version of the game, so as to avoid technicalities. In particular:
\begin{itemize}
\item We do not have to check that Alice is leaving Bob with legal moves.
\item We can without loss of generality assume that Alice always chooses $\varepsilon_{i,j} = \beta\rho(B_i)$, since this places the maximum restriction on Bob's balls.
\item If Alice wins the $k$-dimensional $(\beta,p)$-game, then she automatically wins the $(\beta',p')$-game whenever $\beta'\geq \beta$ and $p'\geq p$.
\end{itemize}

\begin{lemma}
\label{PHAWequivalence}
For each $0\leq k\leq d - 1$, a set $A\subseteq \Reals^d$ is $k$-dimensionally percentage winning if and only if it is $k$-dimensionally absolute winning.
\end{lemma}
\begin{proof}
The backwards direction being trivial, let us suppose that $A$ is $k$-dimensionally percentage winning. Then for some $0 < p < 1$ and $\beta_0 > 0$, $A$ is $k$-dimensionally $(\beta,p)$-winning for all $0 < \beta < \beta_0$. We claim first that this is true for all $0 < p < 1$. Indeed, fix $0 < p' < 1$ and $0 < \beta' < \beta_0$, and let us play the $k$-dimensional $(\beta',p')$-game. Let $m\in\Nat$ be large enough so that $(1 - p')^m \leq 1 - p$, and let $\beta = (\beta')^m$. Then Alice can win the $k$-dimensional $(\beta,p)$-game. Translate Alice's strategy in the $(\beta,p)$-game into a strategy for the $(\beta',p')$-game by replacing each move that Alice makes in the $(\beta,p)$-game with a sequence of $m$ moves in the $(\beta',p')$-game. Specifically, if Alice deletes a set of $N_i$ $k$-planes in the $(\beta,p)$-game, then we will let her spend $m$ moves deleting the same set of $k$-planes in the $(\beta',p')$-game. Details are left to the reader.

Fix $0 < \beta < 1/3$, and let us play the $k$-dimensional $\beta$-absolute game. Fix any $0 < \beta' < \beta$ and consider the set
\[
X \df \{\cl^{(\beta')}\cap B:\cl\subseteq\Reals^d\text{ is an affine $k$-plane}\}\setminus\{\varnothing\},
\]
where $B = \overline{B}(\0,1)$ is the closed unit ball in $\Reals^d$. Notice that $X$ is compact in the Hausdorff metric.

For each affine $k$-plane $\widetilde{\cl}\subseteq\Reals^d$, consider the set
\[
U_{\widetilde{\cl}} \df \{\cl^{(\beta')}\cap B\in X:\cl^{(\beta')}\cap B\subseteq\mathrm{Int}(\widetilde{\cl}^{(\beta)})\}.
\]
Here $\mathrm{Int}(\widetilde{\cl}^{(\beta)})$ is the interior of $\widetilde{\cl}^{(\beta)}$. It is not hard to see that $U_{\widetilde{\cl}}$ is an open subset of $X$ containing $\widetilde{\cl}^{(\beta')}\cap B$.
Thus, $(U_{\widetilde{\cl}})_{\widetilde{\cl}}$ is an open cover of $X$. Since $X$ is compact, there exists a finite subcover. Let $N$ be the size of such a subcover $(U_{\widetilde{\cl}_i})_{i = 1}^N$, and let $p = 1/N$. As described in the first paragraph, $A$ is $k$-dimensionally $(\beta',p)$-winning.

Consider a strategy for Alice to win the $k$-dimensional $(\beta',p)$-game by forcing the intersection point to land in $A$. We will translate every move that Alice makes in this strategy into a move in the $k$-dimensional $\beta$-absolute game, in such a way so that every legal move that Bob can make in the $k$-dimensional $\beta$-absolute game, which
we will call Game 1, is also a legal move in the other game, which we will call Game 2. 
Clearly, this implies that $A$ is $k$-dimensionally $\beta$-absolute winning.

Suppose that Bob has just made the move $B_k = B(\x,\rho)$ in Game 2. Since Alice has a winning strategy in this game, she makes the move $(\cl_j^{(\beta'\rho)})_{j = 1}^{N_k}$, for some $N_k\in\Nat$. Let $T:\Reals^d\rightarrow\Reals^d$ be an affine similarity with contraction ratio $\rho$ such that $T(\0) = \x$. Now each set $\cl_j^{(\beta'\rho)}$ can be viewed as $T(\mathcal{A}_j^{(\beta')})$ for some $k$-plane $\mathcal{A}_j$. In particular, $\mathcal{A}_j^{(\beta')}\cap B\in X$, so since $(U_{\widetilde{\cl}_i})_{i = 1}^N$ is a cover of $X$ there exists some $i_j = 1,\ldots,N$ such that
\begin{equation}
\label{ijbetacontained}
\mathcal{A}_j^{(\beta')} \subseteq \widetilde{\cl}_{i_j}^{(\beta)}.
\end{equation}
Now we can write $\{1,\ldots,N_k\} = \bigcup_{i = 1}^N \{j:i_j = i\}$, and applying the pigeonhole principle, there exists some $i = 1,\ldots,N$ such that
\begin{equation}
\label{pNk}
\#\{j = 1,\ldots,N_k: i_j = i\} \geq \frac{N_k}{N} = p N_k.
\end{equation}
We can now describe Alice's strategy in Game 1. Her strategy will be to remove the $k$-plane $T(\widetilde{\cl}_i^{(\beta)}) = T(\widetilde{\cl}_i)^{(\beta\rho)}$, where $i$ is any value satisfying (\ref{pNk}). To complete the proof, we need to show that any legal move that Bob can make in Game 1 is also legal in Game 2. Since $\beta' < \beta$, it is clear that the size of Bob's ball is not an obstacle. Suppose that $B_{k + 1}$ is any move Bob makes that avoids the set $T(\widetilde{\cl}_i)^{(\beta\rho)}$; i.e. $B_{k + 1}$ is legal in Game 1. Then by (\ref{ijbetacontained}), we have that $B_{k + 1}$ also avoids the sets $\mathcal{A}_j^{(\beta')}$ for all $j = 1,\ldots,N_k$ such that $i_j = i$. But by (\ref{pNk}), this constitutes at least $p N_k$ sets that Bob is avoiding, so the move $B_{k + 1}$ is also legal in Game 2.
\end{proof}

Thus, the HAW and HPW classes are identical, and although some of our strategies below
are given for the hyperplane percentage game, the sets are all in fact shown to be HAW.
The main advantage of the hyperplane game over the classical Schmidt's game
is that it produces a class of sets which is closed under diffeomorphisms.
More precisely, we have the following theorem, which was proved in \cite{BFKRW}.

\begin{thm}
\label{stability}
 Let $S \subseteq \rd$ be  $k$-dimensionally  absolute  winning, 
 $U \subseteq \rd$ open, and
 $f: U \to \rd$ 
a $C^1$ nonsingular map.
Then $f^{-1}(S) \cup U^c$ is
$k$-dimensionally absolute winning.
\end{thm}

\section{Hyperplane diffuse sets and Ahlfors regular measures}
\label{measures}
In this section we consider subsets of $\rd$ which,
%are hyperplane diffuse and thus, 
when used as playgrounds for Schmidt's game,
permit strategies which involve avoiding neighborhoods of specified hyperplanes.
\ignore{In fact, in the literature these classes of measures were defined
prior to the geometric condition of diffuseness, and their supports}
These were the first fractals on which Schmidt's game was played, in \cite{F}.
We begin with a definition introduced in \cite{BFKRW}:
\begin{defn}
\label{diffuseness}
A closed set $K\subset\rd$ is said to be {\sl $k$-dimensionally $\beta$-diffuse\/} (here $0 \leq k < d$, $0<\beta<1$) if there exists 
$\rho_K > 0$ such that for any $0 < \rho \le \rho_K$, $\x\in K$,
and any $k$-dimensional affine subspace $\cl$, there exists $\x' \in K$ such that
$$
\x' \in B(\x,\rho)\ssm \cl^{(\beta\rho)}.
$$
We say that $K$ is 
{\sl $k$-dimensionally diffuse\/} if it is $k$-dimensionally $\beta_0$-diffuse for some $\beta_0 < 1$
(and hence for all $\beta \le \beta_0$). When $k = d-1$, this property will be referred to as {\sl hyperplane diffuseness}; clearly it implies $k$-dimensional diffuseness for all $k$.
\end{defn}

Whenever a set is hyperplane diffuse, every HAW set will be winning for Schmidt's game, and every winning set will have positive dimension (see  \cite{BFKRW}). However, to obtain \emph{full} dimension, 
and hence strong $C^1$ incompressibility, we will need a further 
measure-theoretic assumption on $K$, Ahlfors regularity.

We say a locally finite Borel measure $\mu$ is {\sl $\delta$-Ahlfors regular}
if there exist positive constants $c_1, c_2,$ and $\rho_0$ such that
\begin{equation}
\label{Ahlfors}
c_1\rho ^{\delta}\leq\mu\big(B(\x,\rho)\big)\leq c_2\rho^{\delta}\,
						\quad\forall\,\x\in \supp \, \mu,\ \forall\,0 < \rho<\rho_0 \,.
\end{equation}
						
\noindent Again we will often refer to a measure as simply {\sl Ahlfors regular} when
the parameter is immaterial for our purposes.

The following theorem gives a large class of examples of hyperplane diffuse sets supporting Ahlfors regular measures which includes e.g. the Cantor middle-thirds set and the Sierpinski carpet:
\begin{pro}
\label{propositiondecayingregular}
Let $\{u_1,\ldots,u_m\}$ be a family of contracting similarities of $\R^d$ satisfying the open set condition, and let $K$ be the limit set of this family. If $K$ is not contained in any affine hyperplane, then the Hausdorff measure in the appropriate dimension restricted to $K$ is Ahlfors regular and absolutely decaying. In particular, $K$ is hyperplane diffuse.
\end{pro}
\begin{proof}
The proof of Ahlfors regularity can be found in \cite{Mat} (Theorem 4.14), and absolute decay was proven in \cite{KLW} (Theorem 2.3), with the assumption that $K$ is not contained in any affine hyperplane replaced by the assumption that there is no finite collection of proper affine subspaces $\cl_1,\ldots,\cl_k$ which is invariant under the similarities $u_1,\ldots,u_m$. In fact, these seemingly different assumptions are actually equivalent. Indeed, suppose that there is such a collection $\cl_1,\ldots,\cl_k$. If $\bigcap_{i = 1}^k \cl_i\neq\varnothing$, then $K\subseteq \bigcap_{i = 1}^k \cl_i$ since this subspace is invariant under the family of similarities. On the other hand, if $\bigcap_{i = 1}^k \cl_i = \varnothing$, then there exists some $N < k$ such that $\bigcap_{i = 1}^N \cl_i\neq\varnothing$ but $\bigcap_{i = 1}^{N + 1} \cl_i = \varnothing$. In particular
\begin{equation}
\label{distancecontradiction}
d\left(\bigcap_{i = 1}^N \cl_i,\cl_{N + 1}\right) > 0.
\end{equation}
On the other hand, since $u_1$ must permute the subspaces $\cl_1,\ldots,\cl_k$, its iterate $u_1^{k!}$ must leave each subspace invariant. But then $u_1^{k!}$ must preserve the distance (\ref{distancecontradiction}) above, which is a contradiction since $u_1^{k!}$ is a strict contraction.
\end{proof}

We remark that $\Bad_{\0}(2,1)$ is not incompressible on hyperplane diffuse sets, as the following example illustrates:

\ignore{
The following result appears in \cite{F}.

\begin{lem}  \label{pl dim bound} \cite[Theorem 5.1]{F} Let $\mu$ be
  $\delta$-Ahlfors regular, let $K = \supp\,\mu$, and let $S\subseteq
  \rd$ be winning on $K$. Then $\dim(S\cap K \cap U)  = \delta =
  \dim(K \cap U)$ for every open set $U \subseteq \rd$ with $U \cap K
  \neq \varnothing$. 
\end{lem}

Since the HAW property clearly implies the winning property, this lemma combined
with Theorem \ref{stability} and Lemmas \ref{P-win inher} and \ref{P-intersect} yields the following.

\begin{thm}
If $\mu$ is absolutely decaying and Ahlfors regular, then any HAW set is strongly $C^1$-incompressible on $\supp\mu$.
\end{thm}
}

\begin{exa}
\label{counterexample}
There exists a hyperplane diffuse set $K\subseteq \R^2$ which supports an Ahlfors regular measure, and a bi-Lipschitz map $\Phi:\R^2\rightarrow\R^2$ such that $K\subseteq\Phi(\cl)$, where $\cl$ is the $x$-axis. In particular, $\R^2\setminus \cl$ is not incompressible on $K$, and
hence neither is $\Bad_{\0}(2,1) \subseteq \R^2 \setminus \cl$.
\end{exa}

%alternative exposition about Theorem 1.2 failing for regular incompressibility
\ignore{
Note that since $\Bad_{\0}(2,1)$ is contained in $\R^2\setminus \cl$, it follows that $\Bad_{\0}(2,1)$ is not incompressible on $K$, which demonstrates that Theorem \ref{maintheorem} cannot be improved by replacing ``strongly $C^1$ incompressible'' by ``incompressible''.
}

Also, note that $\R^2\setminus\cl$ is HAW, whereas $\Phi(\R^2\setminus\cl)$ cannot be HAW since $\Phi(\R^2\setminus\cl)$ does not intersect $K$ with full dimension (in fact it does not intersect it at all). Thus, in contrast with Schmidt's winning property, the HAW property is not preserved under bi-Lipschitz self-maps of the playground.

\begin{proof}
Let $K$ be the limit set of the family of contracting similarities $u_i:\R^2\rightarrow\R^2$ defined by
\begin{align*}
u_0(x,y) &:= \left(\frac{x}{5},\frac{y}{5}\right)\\
u_1(x,y) &:= \left(\frac{2 + x}{5},\frac{4 + y}{5}\right)\\
u_2(x,y) &:= \left(\frac{4 + x}{5},\frac{y}{5}\right).
\end{align*}
The open set condition is verified using the set $(0,1)\times(0,1)$. Since $K$ contains the points $(0,0)$, $(1/2,1)$, and $(1,0)$ (the fixed points of $u_0$, $u_1$, and $u_2$, respectively), it follows that $K$ is not contained in any affine hyperplane. Thus by Proposition \ref{propositiondecayingregular}, $K$ is hyperplane diffuse and supports an Ahlfors regular measure.

We claim next that the slope of any line that intersects $K$ in at least two points is at most $5$. This is obvious if the line intersects $K$ in two sets of the form $u_i([0,1]\times[0,1])$. A scaling argument proves the general case.

Thus, $K$ is the graph of a $5$-Lipschitz function on some closed subset of $[0,1]$. By linear interpolation, this function can be extended to a $5$-Lipschitz function $f':[0,1]\rightarrow[0,1]$, which can then be extended to another $5$-Lipschitz function $f:\R\rightarrow[0,1]$. Note that $K$ is contained in the graph of $f$. Let us now define the function $\Phi:\R^2\rightarrow\R^2$ by
\[
\Phi(x,y) = (x,y + f(x)).
\]
Then $\Phi(\cl)$ is exactly the graph of $f$, so $K\subseteq \Phi(\cl)$. Furthermore, since $f$ is $5$-Lipschitz we have that $\Phi$ is $6$-bi-Lipschitz.
\end{proof}

\section{Proof of Theorem \ref{BadAHAW}}
\label{BadAproof}
We obtain Theorem \ref{BadAHAW} by proving a more general theorem concerning what we will
call escaping sets. Suppose that $\cm = (M_k)_{k\in\Nat}$ is a sequence of
$M\times N$ matrices with real entries and that $\cz = (Z_k)_{k\in\Nat}$ is a sequence of subsets of $\R^M$. Following \cite{BFK}, we define
\begin{equation}
\label{defE3}
\tilde E(\cm,\cz) = \{\x \in\Reals^N : \inf_{k\ge 0} \dist(M_k\x,Z_k) > 0\}.
\end{equation}
We will abuse notation slightly and write $\tilde E(\cm,Z) = \tilde E(\cm,\cz)$ if $\cz = (Z)_{k\in\Nat}$ is a constant sequence. Note that in this case, $\tilde E(\cm,\cz)$ consists of the points $\x\in\Reals^N$ whose orbit $\{ M_k \x\}$ 
under the sequence $\cm$ of linear maps remains some fixed distance from the
set $Z$.

In the case $M = N = 1$, $\cm$ is a sequence of reals
and it was shown by A.\ D.\ Pollington in \cite{P} and B.\ de Mathan in \cite{Ma} that
if this sequence is lacunary (i.e. $\inf_k \frac{M_{k+1}}{M_k} > 1$), then $\dim(\tilde E(\cm,\Int)) = 1$.
This was improved in \cite{BBFKW}, where it was shown that for any $y \in \Reals$
and any lacunary sequence $\cm$ of reals, $\tilde E(\cm,y + \Int)$ is a winning set in Schmidt's game.

In \cite{BFK}, the notion of a lacunary sequence was generalized to a sequence of matrices; specifically, a sequence of matrices $\cm$ is said to be lacunary if $\inf_k \frac{\|M_{k+1}\|_\op}{\|M_k\|_\op} > 1$, where
$\|\cdot\|_\op$ stands for the operator norm. This paper also introduced the notion of a \emph{uniformly discrete} sequence of sets, which is a sequence $\cz$ such that
\[
\inf_{k\in\Nat} \inf_{\substack{\x,\y\in Z_k \\\text{distinct}}}\|\x - \y\| > 0.
\]
It was shown that if $\cm$ is a lacunary sequence of matrices and if $\cz$ is a uniformly discrete sequence of sets, then $\tilde E(\cm,\cz)$ is winning on the support of any absolutely friendly measure. Finally, a relation between the sets $\tilde E(\cm,\cz)$ and $\Bad_A(M,N)$ was established: specifically, it was proven that
\begin{equation}
\label{BFKrelation}
\tilde E(\cy,\Int) \subseteq \Bad_A(M,N)
\end{equation}
for a certain lacunary sequence $\cy$ of $1\times M$ matrices which depends on $A$. Since $\cz = (\Int)_{k\in\Nat}$ is clearly a uniformly discrete sequence of sets, this implies that $\Bad_A(M,N)$ is winning on the support of any absolutely friendly measure.

We generalize this result by proving that $\tilde E(\cm,\cz)$ is HAW for any lacunary sequence of matrices $\cm$ and any uniformly discrete sequence of sets $\cz$. In particular, by (\ref{BFKrelation}) it follows that $\Bad_A(M,N)$ is HAW, which proves Theorem \ref{BadAHAW}.

\begin{thm}
\label{dani general}
If $\cm = (M_k)_{k\in\Nat}$ is a lacunary sequence of $M\times N$ matrices with real entries and if $\cz = (Z_k)_{k\in\Nat}$ is a uniformly discrete sequence of subsets of $\R^M$, then $\tilde E(\cm,\cz)$ is HAW.
\end{thm}

\begin{proof}%[Proof of Theorem \ref{dani general}]
For each $k\in\Nat$ let $t_k \df \|M_k\|_{\op}$
and let $\vbf_k$ be a unit vector satisfying
$$ \|M_k\vbf_k\| =t_k.$$
 Let 
\begin{align} \label{ratio}
Q &\df \inf_{k\in\Nat} \frac{t_{k+1}}{t_k} > 1\\ \label{deltadef}
\delta &\df \inf_{k\in\Nat} \inf_{\substack{\x,\y\in Z_k \\\text{distinct}}}\|\x - \y\| > 0.
\end{align}
To show that $\tilde E(\cm,\cz)$ is HAW, we will demonstrate a strategy for Alice to win the hyperplane percentage game; specifically, for every $0 < \beta < 1$ we will demonstrate a strategy for the hyperplane $(\beta,1/2)$-game. Fix such a $\beta$, and choose $n$ large enough so that 
\begin{equation}
\label{large n}
\beta^{-r} \leq Q^n\text{, where } r =  \lfloor \log_{2} n\rfloor +1\,.
\end{equation}
Alice's strategy will be to divide the game into windows. For each $j,k\in\Nat$, we will say that $t_k$ \emph{lies in the $j$th window} if
\begin{equation}
\label{indices}
\beta^{-r (j-1)}t_1 \leq t_k  < \beta^{-r j}t_1.
\end{equation}
Note that every $t_k$ lies in exactly one window. On the other hand, if $j\in \Nat$ is fixed, then by (\ref{ratio}) and (\ref{large n}), 
there are at most $n$ indices $k$ for which $t_k$ lies in the $k$th window.

By playing arbitrary moves if needed, 
we may assume without loss of generality that Bob's first move $B_1$ has radius
\begin{equation}
\label{rho small}
\rho_1 < \beta^r\delta t_1/4.
\end{equation}
Now Alice will divide the game into stages in the following manner: The $j$th stage begins when the radius of Bob's ball $B_j$ satisfies
\[
\rho(B_j) \leq \beta^{r(j - 1)}\rho_1.
\]
Note that the first stage therefore begins with Bob's initial ball $B_1$. Also note that Bob must make at least $r$ moves in each stage.

Suppose that Bob has just played the ball $B_j$, beginning stage $j$. Fix $k\in\Nat$ such that $t_k$ lies in the $j$th window.
For any $\x \in\Reals^N$,
$ \| \x \| 
	\geq \frac{1}{t_k}\|M_k(\x)\|$.
Thus, if $\y_1,\y_2$ are two different points in  $Z_k$, then by (\ref{indices}) and (\ref{rho small})
\begin{equation}
\label{dist2}
\dist\Big(M_k^{-1}\big(B(\y_1,\delta/4)\big),M_k^{-1}\big(B(\y_2,\delta/4)\big)\Big) 
	\geq \frac{\delta/2}{t_k}
		> \frac{\delta}{2t_1}\beta^{rj} \geq 2\rho(B_j).
\end{equation}
Therefore $B_j$ intersects with at most one set of the form $M_k^{-1}\big(B(\z,\delta/4)\big)$, where
$\z\in Z_k$. Hence, for each $k$ satisfying (\ref{indices}),
\begin{equation}
\label{Z preimage}
B_j \cap M_k^{-1}(Z_k^{(c)}) 
	\subseteq M_k^{-1}\big(B(\y_k,c)\big) \text{ for some }\y_k\in Z_k,
\end{equation}
where
\begin{equation}
\label{defc}
c \df \min\left(\beta^{r + 1}\rho_1 t_1,\frac{\delta}{4}\right) > 0.
\end{equation}
(The reason for this value of $c$ will be clear shortly.) We will now show that the preimage of such a ball is contained in a ``small enough'' neighborhood of some hyperplane.
Toward this end,
let $V\subseteq\Reals^M$ be the hyperplane perpendicular to $M_k\vbf_k$ and passing through
$\0$. Then $$W \df M_k^{-1}(V)$$ is a hyperplane in $\Reals^N$ passing through $\0$.

If $\x\not\in W^{(c/t_k)}$, then $\x = \w + \eta \vbf_k$ for some
$\eta > c/t_k$ and $\w\in W$, thus
$$\|M_k\x\| = \|M_k\w + M_k\eta\vbf_k\| \geq \eta\|M_k\vbf_k\| = t_k\eta > c\,.$$
Hence, $M_k^{-1}\big(B({\mathbf 0},c)\big) \subseteq W^{(c/t_k)}$, which clearly implies $M_k^{-1}\big(B(\y_k,c)\big) \subseteq \cl^{(c/t_k)}$ where $\cl = M_k^{-1}(\y_k) + W$ is an affine hyperplane.
By (\ref{indices}) and (\ref{defc}),
\[
\frac{c}{t_k} \leq \beta^{rj + 1}\rho_1 \leq \beta^r\rho(B_j) \df \zeta.
\]
Therefore, by (\ref{Z preimage}),
\begin{equation}
%\label{thickness}
\displaystyle\bigcup_{t_k \text{ in the $j$th window }} B_j\cap M_k^{-1}\big(Z_k^{(c)}\big)	\subseteq \bigcup_{i=1}^n \cl_i^{(\zeta)},
\end{equation}
where $\cl_i$ are hyperplanes.
Alice will choose these $n$ hyperplane neighborhoods as her next turn,
and on her subsequent $r - 1$ turns choose those which remain after intersecting with Bob's ball. The legality of the moves is guaranteed by the definition of $\zeta$. Also, as observed above, Bob must make at least $r$ moves in the $j$th stage and so these moves do not interfere with the next stage.

At the end of stage $j$, therefore, we have that the number of hyperplane-neighborhoods $\cl_i^{(\zeta)}$ which intersect Bob's ball $B_{j + 1}$ is at most $2^{-r}n$, but by (\ref{large n}) this number is strictly less than $1$. Thus $B_{j + 1}$ will be disjoint from the sets $\cl_i^{(\zeta)}$. Thus for $t_k$ in the $j$th window, we have
$$B_{j + 1} \cap M_k^{-1}\big(Z_k^{(c)}\big) = \varnothing\,.$$
We conclude that $\dist(M_k\x, Z_k) \geq c$ for any $\x \in B_{j + 1}$, 
which implies the desired statement.

\end{proof}

\section{Outline of the proof of Theorem \ref{maintheorem}}
\label{outline}
Let $H=M\cdot N$ and $L=M+N$. We shall be playing the game on $\RH$ 
where we identify points in $\RH$ with 
$M\times N$ real matrices.
For $k\in\mathbb{N}$ we denote the $k$th ball chosen by Bob by 
$B(k)$.  Let $\rho=\rho(B(0))$, and set
$
\sigma = \max\{\|X\|\ : X\in B(0) \}. 
$
We assign boldface lower case letters ($\x$, $\y$, etc.) 
to denote points in $\mathbb{R}^N$ and $\mathbb{R}^M$ while boldface upper case letters 
($\X$, $\Y$, $\B$, etc.) denote points in $\mathbb{R}^L$.
Finally, upper case letters ($A$,$X$,$Y$, etc.) denote points in $\mathbb{R}^H$.\\

For any 
\begin{center}
$A=\left(\begin{matrix}
\gamma_{11} & . & . & . & \gamma_{1N}\\
. & . & . & .& . \\
. & . & . & .& . \\
. & . & . & .& . \\
\gamma_{M1} & . & . & . & \gamma_{MN}
\end{matrix}\right)$
\end{center}
let 
\begin{align*}
\A_1&=(\gamma_{11},...,\gamma_{1N},1,0,...,0) & \B_1&=(\gamma_{11},...,\gamma_{M1},1,0,...,0)\\
\A_2&=(\gamma_{21},...,\gamma_{2N},0,1,...,0) & \B_2&=(\gamma_{12},...,\gamma_{M2},0,1,...,0)\\
&\ldots & &\ldots\\
\A_M&=(\gamma_{M1},...,\gamma_{MN},0,0,...,1) & \B_N&=(\gamma_{1N},...,\gamma_{MN},0,0,...,1).
\end{align*}

\vspace{2mm}

\noindent For $\X\in \R^L$, let $\x\in\R^N$ be the projection of $\X$ onto the first $N$ coordinates. Similarly, for $\Y\in\R^L$, let $\y\in\R^M$ be the projection of $\Y$ onto the first $M$ coordinates.

Set
\begin{equation}
{\mathcal{A}}(\X)=(\A_{1}\cdot\X,\ldots ,\A_{M}\cdot\X)
\label{special A}
\end{equation}
and
\begin{equation}
{\mathcal{B}}(\Y)=(\B_{1}\cdot\Y,\ldots ,\B_{N}\cdot\Y).
\label{special B}
\end{equation}

\noindent We notice that a matrix $A$ lies in 
$\Bad_\0(M,N)$ if and only if there exists a constant $c$ such that 
for all $\X\in\Z^L$ with $\x\neq 0$,

\begin{equation}
\|\x\|^{N}\cdot\|{\mathcal{A}}(\X)\|^{M} > c.
\label{BA def}
\end{equation}

\vspace{2mm}

\noindent Let $\lambda = N/L$. Given $1<R\in\mathbb{R}$, let $\delta =R^{-NL^2}$  and  $\delta^T=R^{-ML^2}$, and consider the inequalities
\begin{align} \label{list1}
0<\xnm&<\delta R^{M(\lambda +i)}\\ \label{list2}
\Anm&<\delta R^{-N(\lambda +i)-M}\\ \label{list3}
0<\ynm &<\delta^{T} R^{N(1 +j)}\\ \label{list4}
\Bnm &<\delta^{T} R^{-M(1 +j)-N}.
\end{align}

\begin{observation}
\label{observationwindows}
Fix $A\in\Reals^H$, and suppose that for each $i\in\Nat$, the system of equations \textup{(\ref{list1})}, \textup{(\ref{list2})} has no integer solution $\X$. Then $A\in \Bad_\0(M,N)$.
\end{observation}
\begin{proof}
For each $\X\in\mathbb{Z}^L$ with $\x\neq 0$, we have $\|\x\| \geq 1 \geq \delta R^{M(\lambda - 1)}$, so there exists a unique $i\in\Nat$ such that
\[
\delta R^{M(\lambda +i - 1)}\leq \xnm<\delta R^{M(\lambda +i)}.
\]
Since $\X$ is a solution to (\ref{list1}), $\X$ cannot be a solution to (\ref{list2}), i.e.
\[
\Anm \geq \delta R^{-N(\lambda +i)-M}.
\]
Multiplying the two lower bounds after raising them to appropriate powers gives (\ref{BA def}) with $c = \delta^L R^{-ML}$.
\end{proof}

\begin{remark}
The absence of integer solutions to the system of equations (\ref{list3}), (\ref{list4}) is not needed to show that a matrix $A$ is badly approximable. However, in order to show that Alice can play in a way such that (\ref{list1}), (\ref{list2}) have no solutions, it is necessary for her to first play so that (\ref{list3}), (\ref{list4}) have no solutions when $j = i - 1$. Dually, in order to show that Alice can play in a way such that (\ref{list3}), (\ref{list4}) have no solutions, it is necessary for her to first play so that (\ref{list1}), (\ref{list2}) have no solutions when $i = j$. 
\end{remark}

\noindent We recall the following propositions due to Schmidt (Lemmas 1 and 2 in \cite{S2}):\\

\begin{pro}
There exists a constant $R_1=R_1(M,N,\sigma )$ such that for every $i\in \mathbb{N}$ and for every $R\geq R_1$,
if a ball $B$ satisfies
\begin{equation}
\label{rhoB1}
\rho (B)<R^{-L(\lambda+i)}
\end{equation} 
and if for all $A\in B$ the system of equations \textup{(\ref{list1})}, \textup{(\ref{list2})}
has no integer solution $\X$,
then the set of all vectors $\Y\in\mathbb{Z}^L$ satisfying \textup{(\ref{list3})} with $j = i$
such that there exists $A\in B$ satisfying \textup{(\ref{list4})} with $j = i$
spans a subspace of $\mathbb{R}^L$ whose dimension is at most $N$.
\label{Schmidt first lemma}
\end{pro}

\begin{pro}
\label{last lemma}
There exists a constant $R_2=R_{2}(M,N,\sigma )$ such that for every $j\in \mathbb{N}$ and for every $R\geq R_2$,
if a ball $B$ satisfies 
\begin{equation}
\label{rhoB2}
\rho (B)<R^{-L(1+j)}
\end{equation}

\noindent and if for all $A\in B$ the system of equations \textup{(\ref{list3})}, \textup{(\ref{list4})}
has no integer solution $\Y$,
then the set of all vectors $\X\in\mathbb{Z}^L$ satisfying \textup{(\ref{list1})} with $i = j + 1$ such that there exists $A\in B$ satisfying \textup{(\ref{list2})} with $i = j + 1$
spans a subspace of $\mathbb{R}^L$ of dimension at most $M$.
\label{Schmidt second lemma}
\end{pro}

Let us now say a few words about the proof of Theorem \ref{maintheorem}. Suppose that a game has already been played; for each $i\in\Nat$, let $k_i\in\Nat$ be the minimal number such that (\ref{rhoB1}) holds with $B = B(k_i)$, and for each $j\in\Nat$, let $h_j\in\Nat$ be the minimal number such that (\ref{rhoB2}) holds with $B = B(h_j)$. Alice will try to play the game in such a way so that for all $A\in B(k_i)$, the system of equations (\ref{list1}), (\ref{list2}) has no integer solution $\X$ and so that for all $A\in B(h_j)$, the system of equations (\ref{list3}), (\ref{list4}) has no integer solution $\Y$. The following lemma states that she can always continue to do this if $R$ is sufficiently large:

\begin{lemma}
\label{lemmareduced}
If $R\in\Reals$ is sufficiently large, then the following hold:
\begin{enumerate}[i)]
\item When the game is at stage $k_i$, if, for all $A\in B(k_i)$, the system of equations \textup{(\ref{list1})}, \textup{(\ref{list2})} has no integer solution $\X$, and the system of equations \textup{(\ref{list3})}, \textup{(\ref{list4})} with $j = i - 1$ has no integer solution $\Y$, then Alice has a strategy ensuring that when the game reaches stage $h_i$, then for all $A\in B(h_i)$ the system of equations \textup{(\ref{list3})}, \textup{(\ref{list4})} with $j = i$ will have no integer solution $\Y$.
\item Dually, when the game is at stage $h_j$, if, for all $A\in B(h_j)$, the system of equations \textup{(\ref{list3})}, \textup{(\ref{list4})} has no integer solution $\Y$, and the system of equations \textup{(\ref{list1})}, \textup{(\ref{list2})} with $i = j$ has no integer solution $\X$, then Alice has a strategy ensuring that when the game reaches stage $k_{j+1}$, then for all $A\in B(k_{j+1})$ the system of equations \textup{(\ref{list1})}, \textup{(\ref{list2})} with $i = j + 1$ will have no integer solution $\X$.
\end{enumerate}
\end{lemma}

\noindent We will prove the first claim of this lemma; the proof of the second claim is the same but dual.

\noindent The proof of Lemma \ref{lemmareduced} will be divided into two parts: in the first part, we will discover the consequences of the hypotheses of Lemma \ref{lemmareduced}; in the second part, we will describe how Alice will play.

\begin{proof}[Proof of Lemma \ref{lemmareduced}, Part One]
For convenience of notation let $B = B(k_i)$. Suppose that $R\geq R_1$. Then by Proposition \ref{Schmidt first lemma}, the dimension of the subspace spanned by the set
\begin{equation}
\label{setS}
S \df\{\Y\in\mathbb{Z}^L:\text{there exists $A\in B$ satisfying (\ref{list3}), (\ref{list4}) with $j = i$}\}
\end{equation}
is at most $N$. If necessary, extend this subspace to a subspace of dimension $N$. Then choose an orthonormal basis $\mathcal{Y} = \{\Y_1,\ldots,\Y_N\}$. Now suppose that $\Y\in S$, and let $A\in B$ be the corresponding matrix. Note that by hypothesis, $(\Y,A)$ do not satisfy the system of equations (\ref{list3}), (\ref{list4}) with $j = i - 1$. But $(\Y,A)$ do satisfy (\ref{list4}) with $j = i$, which implies (\ref{list4}) with $j = i - 1$. Thus they must not satisfy (\ref{list3}) with $j = i - 1$, i.e.
\[
\|\y\| \geq \delta^{T}R^{N(1+(i-1))}.
\]
Since $\Y\in S$, we can write $\Y=t_1\Y_1+...+t_N\Y_N$ for some real numbers $t_1,\ldots,t_N$. We have that 

\begin{center}
$\delta ^{T}R^{N i}\leq \|\y\|\leq \|\Y\| = \sqrt{t_{1}^{2}+\ldots +t_{N}^{2}}\leq
\sqrt{N}\max(\left|t_1 \right|,...,\left|t_N\right|)$.
\end{center}

\noindent And so,

\begin{equation}
\delta^T\frac{1}{\sqrt{N}}R^{N i}\leq \max(\left|t_1 \right|,...,\left|t_N\right|).
\label{35}
\end{equation}

\noindent On the other hand, we may write out (\ref{list4}) with $j = i$ as

\begin{center}
$\begin{matrix}
\left|t_1(\B_1\cdot\Y_1)+...+t_N(\B_1\cdot\Y_N)\right|<\delta ^{T}R^{-M(1+i)-N}\\
.\\
.\\
.\\
\left|t_1(\B_N\cdot\Y_1)+...+t_N(\B_N\cdot\Y_N)\right|<\delta ^{T}R^{-M(1+i)-N}.
\end{matrix}$
\end{center}

\noindent Let $D = D(A,\mathcal{Y})$ be the determinant of the matrix
\begin{equation}
\label{MAY}
M(A,\mathcal{Y}) \df (\B_u\cdot\Y_v)_{1\leq u,v\leq N},
\end{equation}

\noindent and let $D_{uv} = D_{uv}(A,\mathcal{Y})$ be the cofactor of the entry $\B_u\cdot\Y_v$ in this matrix.
By Cramer's rule we get for every $1\leq v\leq N$

\[
|t_v|\leq \frac{1}{|D|}N\delta ^{T}R^{-M(1+i)-N}\max(|D_{1v}|,...,|D_{Nv}|)
\]

\noindent and in conjunction with (\ref{35}) we get
\[
\delta^T\frac{1}{\sqrt{N}}R^{N i}
\leq \frac{1}{|D|}N\delta ^{T}R^{-M(1+i)-N}\max(|D_{11}|,|D_{12}|,...,|D_{NN}|)
\]

\noindent or

\begin{equation}
\left|D\right|\leq N\sqrt{N}R^{-L(1+i)}\max(\left|D_{11}\right|,
\left|D_{12}\right|,...,\left|D_{NN}\right|).
\label{last equation}
\end{equation}

To summarize, we have proven:
\begin{lemma}
\label{lemmareducedpartone}
Suppose that for all $A\in B(k_i)$, the system of equations \textup{(\ref{list1})}, \textup{(\ref{list2})} has no integer solution $\X$, and the system of equations \textup{(\ref{list3})}, \textup{(\ref{list4})} with $j = i - 1$ has no integer solution $\Y$. Let $\mathcal{Y} = \{\Y_1,\ldots,\Y_N\}$ be the orthonormal basis described above. Then for every $A\in B(k_i)$, if there exists an integer point $\Y\in\mathbb{Z}^L$ satisfying \textup{(\ref{list3})} and \textup{(\ref{list4})} with $j = i$, then $(A,\mathcal{Y})$ satisfies \textup{(\ref{last equation})}.
\end{lemma}

Thus, if Alice can play in such a way so that (\ref{last equation}) \underline{is not satisfied} by any point $A\in B(h_i)$, then when the game reaches stage $h_i$, for all $A\in B(h_i)$ the system of equations (\ref{list3}), (\ref{list4}) with $j = i$ will have no integer solution $\Y$.
The proof of Lemma \ref{lemmareduced} will be continued in the next section.

\renewcommand{\qedsymbol}{}\end{proof}

\section{Reduction of the proof to a technical lemma}
\label{reduction}

In this section, let us fix a set of orthonormal vectors  ${\mathcal{Y}} =\{\Y_1,...,\Y_N\} \subseteq \mathbb{R}^L$. For each $v \in \{0,\ldots,N\}$ and for each $A\in\mathbb{R}^H$, consider the set of $v\times v$ minors of the matrix $M(A,\mathcal{Y})$ defined by (\ref{MAY}). Each minor can be described by a pair of sets $I,J\subseteq \{1,\ldots,N\}$ satisfying $\#(I) = \#(J) = v$. For each such pair, we define the map
\begin{align*}
D_{(I,J)} &:\mathbb{R}^H \rightarrow\mathbb{R}\\
D_{(I,J)}(A) &\df \det\left( \B_{i_k}\cdot\Y_{j_l}\right)_{k,l = 1,\ldots,v},
\end{align*}
where $I = \{i_k: k = 1,\ldots,v\}$ and $J = \{j_l: l = 1,\ldots,v\}$. In other words, $D_{(I,J)}(A)$ is the determinant of the $v\times v$ minor of $M(A,\mathcal{Y})$ described by the pair $(I,J)$. For shorthand, let $\omega = (I,J)$ and let $\Omega_v = \{(I,J): v = \#(I) = \#(J)\}$. The special cases $v = 0$ and $v = -1$ are dealt with as follows: $\Omega_0 = \{\xi\}$ and $D_\xi(A) = 1$ for all $A$, where $\xi = (\varnothing,\varnothing)$; $\Omega_{-1} = \varnothing$.

\noindent Define
\[
\vec{M}_v(A) = \vec{M}_{v,\mathcal{Y}}(A) \df \left(D_\omega(A)\right)_{\omega\in\Omega_v} \in \mathbb{R}^{{\binom{N}{v}}^2},
\]
and for shorthand let $M_v(A) = \|\vec{M}_v(A)\|$. For each ball $B\subseteq \R^H$ let
\[
M_v(B) \df \sup_{A\in B}M_v(A).
\]

\noindent We can now state our main lemma:

\begin{lemma}
For all $0<\beta<\frac13$, for all $\sigma\in\Reals$, and for all $0\leq v \leq N$ there exists
\begin{align*}
\nu_v &= \nu_v(M,N,\beta,\sigma) > 0
\end{align*}

\noindent such that for any $0 < \mu_v \leq \nu_v$ and for any set of orthonormal vectors $\mathcal{Y} = \Y_1,...,\Y_N \subseteq \mathbb{R}^L$, Alice can win the following finite game:

\begin{itemize}
\item Bob plays a closed ball $B\subseteq\mathbb{R}^H$ satisfying $\rho_B \df \rho(B) < 1$ and $\max_{A\in B}\|A\| \leq \sigma$.
\item Alice and Bob play the hyperplane game until the radius of Bob's ball $B_v$ is less than $\mu_v\rho_B$.
\item Alice wins if for all $A\in B_v$, we have
\begin{equation}
\label{inductionhypothesis}
M_v(A) > \nu_v\rho_B M_{v - 1}(B_v).
\end{equation}
\end{itemize}
\label{hard theorem}
\end{lemma}

\noindent The proof of Lemma \ref{hard theorem} will be delayed until Section \ref{sectionhardtheoremproof}. For now we will assume Lemma \ref{hard theorem}, and use it to complete the proof of Lemma \ref{lemmareduced}:

\begin{proof}[Proof of Lemma \ref{lemmareduced}, Part Two]

Let $\nu_N > 0$ be the number guaranteed by Lemma \ref{hard theorem} for $v = N$. Suppose $R\geq R_1$ is large enough so that

\[
R^{-M}\leq \frac{1}{N\sqrt N}\beta\nu_N.
\]

\noindent Now suppose that the game has progressed to stage $k_i$, and let $\rho_{k_i} = \rho(B(k_i))$. If $k_i > 0$, then since $k_i$ is the minimal integer such that (\ref{rhoB1}) is satisfied with $B = B(k_i)$, we have
\[
\rho_{k_i} \geq \beta R^{-L(\lambda + i)}.
\]

\noindent We can ensure that $k_i > 0$ for all $i\geq 0$ by requiring that
\[
R^{-M} \leq \rho_0,
\]
where $\rho_0$ is the radius of Bob's first ball. In particular, we have
 
\[
\mu_N \df \frac{R^{-L(1 + i)}}{\rho_{k_i}} \leq \frac{1}{\beta R^M} \leq \frac{\nu_N}{N\sqrt N}.
\]
\noindent Since the right hand side is bounded above by $\nu_N$, it follows that this is a valid choice of $\mu_N$.

\noindent Let $\mathcal{Y} = \{\Y_1,\ldots,\Y_N\}$ be an orthonormal basis for a subspace of $\R^L$ containing the set $S$ defined by (\ref{setS}).
This sets the stage for the finite game described in Lemma \ref{hard theorem}, which the lemma says Alice can win.
Now, the finite game is played until Bob's ball $B$ satisfies

\[
\rho(B) < \mu_N \rho_{k_i} = R^{-L(1 + i)},
\]

\noindent in other words, the last ball that Bob plays in the finite game is exactly $B(h_i)$. Since Alice wins the finite game, we have that for every $A\in B(h_i)$, (\ref{inductionhypothesis}) holds with $v = N$. In particular, we have $\omega = (\{1,\ldots,N\},\{1,\ldots,N\}) \in \Omega_N$ and so for all $A\in B(h_i)$

\[
|D(A)| > \nu_N \rho_{k_i}M_{N - 1}(B(h_i))
\geq N\sqrt N R^{-L(1 + i)}\max(|D_{11}(A)|,|D_{12}(A)|,\ldots,|D_{NN}(A)|)
\]
\noindent i.e. (\ref{last equation}) is not satisfied for any point $A\in B(h_i)$. On the other hand, fixing $A\in B(h_i)$, we see by Lemma \ref{lemmareducedpartone} that if there exists an integer point $\Y\in\mathbb{Z}^L$ satisfying (\ref{list3}) and (\ref{list4}) with $j = i$, then (\ref{last equation}) would also be satisfied, a contradiction. Thus the system of equations (\ref{list3}), (\ref{list4}) with $j = i$ has no integer solution $\Y$.
\end{proof}

With this lemma, we complete the proof of Theorem \ref{maintheorem}:

\begin{proof}[Proof of Theorem \ref{maintheorem}]
Let $R\in\Reals$ be chosen large enough so that Lemma \ref{lemmareduced} holds.
Note that when $i = 0$, the equation (\ref{list1}) has no integer solution $\X$,
and when $j = 0$ the equation (\ref{list3}) has no integer solution $\Y$.
Thus Alice may make dummy moves until stage $h_0$,
at which point the hypotheses of Lemma \ref{lemmareduced}(ii) hold.
Then Alice has a strategy to ensure that for all $A\in B(k_1)$,
the system of equations (\ref{list1}), (\ref{list2}) with $i = 1$ has no integer solution $\X$.
By continuing in this way, Alice ensures that if $A$ is the intersection point of the balls $(B(k))_{k\in\Nat}$,
then for all $i\in\Nat$, the system of equations (\ref{list1}),
(\ref{list2}) has no integer solution $\X$.
By Observation \ref{observationwindows}, this implies that $A$ is badly approximable.
Thus the set of badly approximable systems of linear forms is HAW.

\end{proof}

\section{Proof of lemma \ref{hard theorem}} 
\label{sectionhardtheoremproof}
\ignore{
Suppose that ${\mathcal{Y}} =\{\Y_1,...,\Y_N\}$ is an orthonormal set of vectors in $\mathbb{R}^L$. For each $I,J\subseteq \{1,\ldots,N\}$ with $v \df \#(I) = \#(J)$, we define the map
\begin{align*}
D_{(I,J)} &:\mathbb{R}^H \rightarrow\mathbb{R}\\
D_{(I,J)}(A) &\df \det\left( \B_{i_k}\cdot\Y_{j_l}\right)_{k,l = 1,\ldots,v},
\end{align*}
where $I = \{i_k: k = 1,\ldots,v\}$ and $J = \{j_l: l = 1,\ldots,v\}$. For shorthand, let $\omega = (I,J)$ and let $\Omega_v = \{(I,J): v = \#(I) = \#(J)\}$. Define 
\[
\vec{M}_v(A) = \vec{M}_{v,\mathcal{Y}}(A) \df \left(D_\omega(A)\right)_{\omega\in\Omega_v} \in \mathbb{R}^{{\binom{N}{v}}^2}.
\]
The special cases $v = 0$ and $v = -1$ are dealt with as follows: $\Omega_0 = \{0\}$ and $D_0(A) = 1$ for all $A$; $\Omega_{-1} = \varnothing$.

If $B$ is a ball, we define
\[
M_\omega(B) \df \sup_{A\in B}|D_\omega(A)|.
\]

If $\omega' = (I',J')\in\Omega_{v'}$ for $v' < v$, we write $\omega' < \omega$ if $I'\subseteq I$ and $J'\subseteq J$. For $v' < v$ fixed let
\[
\Omega_{v',\omega} \df \{\omega'\in\Omega_{v'}:\omega' < \omega\}.
\]
}
The following two lemmas are essentially due to Schmidt, however we include their proofs for completeness:

\begin{lemma}
\label{lemmaschmidt5}
Fix $A\in B$, $v \in \{0,\ldots,N\}$, and $\omega\in\Omega_v$. Then
\begin{align} \label{nabla}
\|\nabla D_\omega(A)\|_\op &\leq v M_{v - 1}(A)\\ \label{nablanabla}
\|\nabla\nabla D_\omega(A)\|_\op &\leq v^2 M_{v - 2}(A).
\end{align}
\end{lemma}
\begin{proof}
Write $\omega = (I,J)$. For any matrix $A'\in \Reals^H$, it is readily computed that
\begin{equation}
\label{derivativeofDIJ}
\nabla_{A'} D_{(I,J)}(A) = \sum_{k,\ell = 1}^v \pm \left[A'\e_{i_k}\cdot \Y_{j_\ell}\right]D_{(I\setminus\{i_k\},J\setminus\{j_\ell\})}(A).
\end{equation}
Here $\e_1,\ldots,\e_N$ are the standard basis vectors for $\Reals^N$. We identify a vector $\x\in\Reals^M$ with its image in $\Reals^L$ under the inclusion map. It follows that
\[
|\nabla_{A'} D_{(I,J)}(A)| \leq M_{v - 1} \sum_{k,\ell = 1}^v \left|A'\e_{i_k}\cdot \Y_{j_\ell}\right|\\
\leq M_{v - 1} \sqrt{v} \sum_{k = 1}^v \|A'\e_{i_k}\|\\
\leq M_{v - 1} v \|A'\|,
\]
yielding (\ref{nabla}).

Differentiating (\ref{derivativeofDIJ}) with respect to some $A''\in \Reals^H$ yields
\[
\nabla_{A''}\nabla_{A'} D_{(I,J)}(A) = \sum_{k,\ell = 1}^v \sum_{k',\ell' = 1}^v \pm \left[A'\e_{i_k}\cdot \Y_{j_\ell}\right]\left[A''[\e_{i_{k'}}]\cdot \Y_{j_{\ell'}}\right]D_{(I\setminus\{i_k,i_{k'}\},J\setminus\{j_\ell,j_{\ell'}\})}(A)
\]
and a similar computation yields (\ref{nablanabla}).
\end{proof}

\begin{lemma}
\label{lemmaschmidt6}
Fix $A\in B$ and $v \in \{0,\ldots,N\}$. There exists constants $\varepsilon_1,\varepsilon_2 > 0$ depending only on $M$, $N$, and $\sigma$ such that if
\[
M_v(A) \leq \varepsilon_1 M_{v - 1}(A)
\]
then
\[
\max_{\omega\in\Omega_v}\|\nabla D_\omega(A)\|_\op > \varepsilon_2 M_{v - 1}(A).
\]
\end{lemma}

\begin{proof}
Let $\omega' = (I',J')\in\Omega_{v - 1}$ be the value which maximizes the expression $|D_{\omega'}(A)|$. Choose any $i_0\in\{1,\ldots,N\}\setminus I'$, $j_0\in\{1,\ldots,N\}\setminus J'$, and let $\omega = (I,J) = (I'\cup\{i_0\},J'\cup\{j_0\})\in\Omega_v$. We will let $A'$ be such that $A'\e_i = \mathbf{0}$ for all $i$ except $i = i_0$, so that (\ref{derivativeofDIJ}) reduces to
\[
\nabla_{A'} D_{(I,J)}(A) = \sum_{\ell = 1}^v \pm \left[A'\e_{i_0}\cdot \Y_{j_\ell}\right]D_{(I\setminus\{i_0\},J\setminus\{j_\ell\})}(A).
\]
We may choose any value for $A'\e_{i_0}$ which lies in $\Reals^M$. In particular, we may let
\[
A'\e_{i_0} = \Y_{j_0} - \sum_{j = 1}^N\left[\Y_{j_0}\cdot \e_{M + i}\right]\B_i
\]
since computation verifies $A'\e_{i_0}\cdot \e_{M + i} = 0$ for all $i = 1,\ldots,N$. Now
\begin{align*}
\nabla_{A'} D_{(I,J)}(A)
&= \sum_{\ell = 1}^v \pm \left[\Y_{j_0}\cdot \Y_{j_\ell}\right]D_{(I\setminus\{i_0\},J\setminus\{j_\ell\})}(A)\\
&\hspace{.3 in}- \sum_{i = 1}^N\left[\Y_{j_0}\cdot \e_{M + i}\right]\sum_{\ell = 1}^v \pm \left[\B_i\cdot \Y_{j_\ell}\right]D_{(I\setminus\{i_0\},J\setminus\{j_\ell\})}(A)\\
&= \pm D_{(I\setminus\{i_0\},J\setminus\{j_0\})}(A)
- \sum_{i = 1}^N\left[\Y_{j_0}\cdot \e_{M + i}\right]D_{(I\cup\{i\}\setminus\{i_0\},J)}(A)
\end{align*}
and thus
\begin{align*}
|\nabla_{A'} D_\omega(A)| &\geq |D_{\omega'}(A)| - M_v(A)\sum_{i = 1}^N\left|\Y_{j_0}\cdot \e_{M + i}\right|\\
&\geq M_{v - 1}(A) - \sqrt{N}M_v(A).
\end{align*}
On the other hand
\[
\|A'\| \leq 1 + \sqrt{N}\sigma
\]
and the lemma follows.
\end{proof}

We now prove Lemma \ref{hard theorem} by induction on $v$. When $v = 0$ the lemma is trivial (By convention we say that $\max(\varnothing) = 0$). Suppose that the lemma has been proven for $v - 1$, and we want to prove it for $v$. Let $\nu_{v - 1} > 0$ be given by the induction hypothesis. Fix $0 < \mu_{v - 1} \leq \nu_{v - 1}$ and $\nu_v > 0$ to be determined. Suppose that we are given $0 < \mu_v \leq \nu_v$ and $\Y_1,...,\Y_N$ a sequence of orthonormal vectors, and let $B$ be the first ball played by Bob in the finite game. By the induction hypothesis, Alice can play in a way such that if $B_{v - 1}$ is the first ball chosen by Bob satisfying $\rho(B_{v - 1}) < \mu_{v - 1}\rho_B$, then for all $A\in B_{v - 1}$ we have
\begin{equation}
\label{DomegaprimeA}
M_{v - 1}(A) > \nu_{v - 1}\rho_B M_{v - 2}(B_{v - 1}).
\end{equation}
We must describe how Alice will continue her strategy so as to satisfy (\ref{inductionhypothesis}). We begin with the following observation:
\begin{claim}
\label{claimABv1}
For all $A\in B_{v - 1}$ we have
\[
M_{v - 1}(B_{v - 1}) \leq v M_{v - 1}(A).
\]
\end{claim}
\begin{proof}
Fix $A' = A + C \in B_{v - 1}$ and $\omega'\in\Omega_{v - 1}$. We use (\ref{nabla}) to bound the right hand side of the mean value inequality:
\begin{align*}
|D_{\omega'}(A') - D_{\omega'}(A)| &\leq \int_{t = 0}^1 |\nabla_C D_{\omega'}(A + tC)|\d t\\
&\leq \|C\|\int_{t = 0}^1 \|\nabla D_{\omega'}(A + tC)\|_\op\d t\\
&\leq \mu_{v - 1}\rho_B (v - 1)M_{v - 2}(B_{v - 1})\\
&\leq \frac{\mu_{v - 1} (v - 1)}{\nu_{v - 1}}M_{v - 1}(A)\\
&\leq (v - 1)M_{v - 1}(A).
\end{align*}
Thus $|D_{\omega'}(A')| \leq v M_{v - 1}(A)$. Taking the supremum over all $\omega'\in\Omega_{v - 1}$ and all $A'\in B_{v - 1}$ completes the proof.
\end{proof}
To complete the proof of Lemma \ref{hard theorem}, we divide into two cases:
\begin{itemize}
\item[Case 1:] $M_v(A) > \varepsilon_1 M_{v - 1}(A)$ for all $A\in B$. In this case, Alice will make dummy moves until $\rho(B_v) < \mu_v\rho_B$. By Claim \ref{claimABv1} we have $M_v(A) > (\varepsilon_1/v)M_{v - 1}(B_{v - 1}) \geq (\varepsilon_1/v)M_{v - 1}(B_v)$, so (\ref{inductionhypothesis}) holds as long as $\nu_v \leq \varepsilon_1/(v\sigma)$.
\item[Case 2:] $M_v(A) \leq \varepsilon_1 M_{v - 1}(A)$ for some $A\in B$. In this case, by Lemma \ref{lemmaschmidt6} we have
\begin{equation}
\label{nablalowerbound}
\|\nabla D_\omega(A)\|_\op > \varepsilon_2 M_{v - 1}(A)
\end{equation}
for some $\omega\in\Omega_v$. Let
\[
L(A') = D_\omega(A) + \nabla_{A' - A}D_\omega(A)
\]
be the linearization of $D_\omega$ at $A$, and let $\cl = L^{-1}(0)$ be its affine kernel. Alice's strategy will be to delete the neighborhood $\cl^{(\beta\rho(B_{v - 1}))}$ of $\cl$, and then make dummy moves until $\rho(B_v) < \mu_v\rho_B$. The gradient condition (\ref{nablalowerbound}) implies that 
\begin{align*}
|L(A')| &\geq \beta\rho(B_{v - 1})\varepsilon_2 M_{v - 1}(A)\\
&\geq \beta^2\mu_{v - 1}\rho_B \varepsilon_2 M_{v - 1}(A)
\end{align*}
for all $A'\in B_{v - 1}\setminus\cl^{(\beta\rho(B_{v - 1}))}$. On the other hand, letting $C = A' - A$, the standard error formula for linearization tells us that
\[ 
|D_\omega(A') - L(A')| \leq\int_{t = 0}^1 (1 - t)|\nabla_C\nabla_C D_\omega(A + tC)|\d t
\]
and combining with (\ref{nablanabla}) and (\ref{inductionhypothesis}) [with $v = v - 1$] gives
\begin{align*}
|D_\omega(A') - L(A')| &\leq \|C\|^2\int_{t = 0}^1 (1 - t)\|\nabla\nabla D_\omega(A + tC)\|_\op\d t\\
&\leq \|C\|^2\int_{t = 0}^1 (1 - t) v^2 M_{v - 2}(B_{v - 1})\d t\\
&\leq (\mu_{v - 1}\rho_B)^2 \frac{v^2}{2}M_{v - 2}(B_{v - 1})\\
&\leq \frac{(\mu_{v - 1}\rho_B)^2}{\nu_{v - 1}\rho_B} \frac{v^2}{2}M_{v - 1}(A)\\
&= \frac{\mu_{v - 1}^2}{\nu_{v - 1}}\rho_B\frac{v^2}{2}M_{v - 1}(A)
\end{align*}
and thus
\[
|D_\omega(A')| \geq \left(\beta^2\varepsilon_2 - \frac{v^2}{2}\frac{\mu_{v - 1}}{\nu_{v - 1}}\right)\mu_{v - 1}\rho_B M_{v - 1}(A).
\]
Letting
\[
\mu_{v - 1} = \frac{\beta^2\varepsilon_2\nu_{v - 1}}{v^2}
\]
we have
\begin{align*}
M_v(A') \geq |D_\omega(A')| &\geq \frac{1}{2}\beta^2\varepsilon_2\mu_{v - 1}\rho_B M_{v - 1}(A)\\
&\geq \frac{1}{2v}\beta^2\varepsilon_2\mu_{v - 1}\rho_B M_{v - 1}(B_v).
\end{align*}
Letting
\[
\nu_v = \frac{1}{2v}\beta^2\varepsilon_2\mu_{v - 1}
\]
we see that (\ref{inductionhypothesis}) is satisfied with $A = A'$. Now $A'$ was an arbitrary element of $B_{v - 1}\setminus\cl^{(\beta\rho(B_{v - 1}))}$. But because of the restriction on Bob's moves, we have $B_v\subseteq B_{v - 1}\setminus\cl^{(\beta\rho(B_{v - 1}))}$; thus (\ref{inductionhypothesis}) holds for every element of $B_v$.
\end{itemize}

\bibliographystyle{alpha}

\end{document}